\documentclass[11pt]{article}
\usepackage{amsmath,amsthm,amssymb,mathtools}
\usepackage{tikz-cd}
\usepackage[margin=1in]{geometry}
\usepackage[expansion=false]{microtype}
\usepackage[colorlinks=true,linkcolor=blue,citecolor=blue,urlcolor=blue]{hyperref}
\hypersetup{pdfauthor={Florian Lengyel},
  pdftitle={A Cartesian Promonoidal Kernel on Delta and a Hadamard Contraction of the Standard Simplex}}

\title{A Cartesian Promonoidal Kernel on \(\Delta\) and a Hadamard Contraction of \(\Delta^n\)}
\author{Florian Lengyel\\
  \small\href{mailto:flengyel@gradcenter.cuny.edu}{\texttt{flengyel@gradcenter.cuny.edu}}}
\date{September 29, 2026}

\newtheorem{proposition}{Proposition}[section]

\theoremstyle{remark}
\newtheorem{remark}[proposition]{Remark}

\newcommand{\D}{\Delta}
\newcommand{\op}{\mathrm{op}}
\newcommand{\id}{\mathrm{id}}
\newcommand{\Set}{\mathbf{Set}}
\newcommand{\Pos}{\mathbf{Pos}}
\newcommand{\boxt}{\mathbin{\boxtimes}}
\newcommand{\starc}{\mathbin{\star}}

\begin{document}
\maketitle

\begin{abstract}
Pointwise multiplication of nondecreasing maps defines simplicial maps
\(H_{p,q}:\Delta^p\times\Delta^q\to\Delta^{pq}\) for nonnegative integers
\(p,q\). We express these maps using Day's promonoidal construction for
the Cartesian product of presheaves, applied to the simplex category.
For \(q=1\), the map is a simplicial homotopy from the constant
\(0\)-vertex map to the identity. We compute its geometric realization
by a piecewise-linear formula in barycentric coordinates and describe
it on a prism triangulation.
\end{abstract}


\section{Introduction}
Pointwise multiplication gives an explicit simplicial homotopy from
the constant \(0\)-vertex map to the identity on every standard simplex.
The simplex category \(\D\) has objects
\mbox{\([n]=\{0<1<\cdots<n\}\)}, for \(n\geq0\), and nondecreasing maps as
morphisms. Its presheaves are simplicial sets. Write
\(\D^n=\D(-,[n])\) for the standard simplicial \(n\)-simplex.
For fixed nonnegative integers \(p,q\), pointwise multiplication of
nondecreasing maps \([m]\to[p]\) and \([m]\to[q]\) commutes with
precomposition. It therefore defines a simplicial map
\(H_{p,q}:\D^p\times\D^q\to\D^{pq}\).
The homotopy is the special case \(q=1\).

Day's construction \cite[\S5, pp.~33--36, Example 5.2]{Day1970}
expresses the Cartesian product of presheaves as a convolution.
We recall that construction and
identify the corresponding map on representables. We then specialize to
\(q=1\) and compute the realization of the resulting simplicial
homotopy. Section~\ref{sec:homotopy} also gives the alternative nerve
description associated with the classical ordered-set diagram.
On the unit interval the realized map is \((x,t)\mapsto\min(x,t)\).
The higher-dimensional formula is expressed by truncating the tail sums
of the barycentric coordinates.

\section{The Cartesian promonoidal kernel}
The kernel describes the Cartesian product of presheaves in terms of
morphisms in the underlying category.
Let \(C\) be a small category and write
\(\widehat C=\mathrm{Fun}(C^{\op},\Set)\) for its category of
set-valued presheaves. All coends below are taken in \(\Set\).
A promonoidal structure on \(C\) consists of functors
\[
 P:C\times C\times C^{\op}\to\Set,
 \qquad J:C^{\op}\to\Set,
\]
together with an associator \(\alpha\), a left unitor \(\lambda\),
and a right unitor \(\rho\). Their components, for objects
\(a,b,c,u\) of \(C\), are isomorphisms of sets
\begin{align*}
 \alpha_{a,b,c;u}:\ 
 \int^{v\in C}P(a,b;v)\times P(v,c;u)
 &\xrightarrow{\;\cong\;}
 \int^{v\in C}P(b,c;v)\times P(a,v;u),\\
 \lambda_{a;u}:\ 
 \int^{v\in C}J(v)\times P(v,a;u)
 &\xrightarrow{\;\cong\;} C(u,a),\\
 \rho_{a;u}:\ 
 \int^{v\in C}P(a,v;u)\times J(v)
 &\xrightarrow{\;\cong\;} C(u,a).
\end{align*}
These maps are natural covariantly in \(a,b,c\), wherever those
variables occur, and contravariantly in \(u\). The associator and
unitors satisfy the pentagon and triangle identities.

A symmetry of this promonoidal structure is a natural family of
isomorphisms
\[
 s_{a,b;u}:P(a,b;u)\xrightarrow{\;\cong\;}P(b,a;u)
\]
satisfying the hexagon identities with \(\alpha\), the unit
identities with \(\lambda,\rho\), and
\[
 s_{b,a;u}\circ s_{a,b;u}=\id_{P(a,b;u)}.
\]
A promonoidal structure equipped with such a symmetry is called
symmetric.

The associated Day convolution is the bifunctor
\[
 \starc:\widehat C\times\widehat C\longrightarrow\widehat C
\]
defined, for \(F,G\in\widehat C\) and \(u\in C\), by
\[
 (F\starc G)(u)
 =\int^{a,b\in C}P(a,b;u)\times F(a)\times G(b).
\]
Its unit object is \(J\); see \cite{Day1970}.
We use the co-Yoneda isomorphism
\[
 \eta_{F,u}:
 \int^{a\in C}C(u,a)\times F(a)
 \xrightarrow{\;\cong\;}F(u),
\]
natural in \(F\in\widehat C\) and \(u\in C^{\op}\);
see \cite[\S3.10, (3.71)]{Kelly1982}.

\begin{proposition}\label{prop:cartesian}
For every small category \(C\), the functors
\[
 P_C(a,b;u)=C(u,a)\times C(u,b),\qquad J_C(u)=\{*\}
\]
define a symmetric promonoidal structure on \(C\). Its Day convolution
is naturally isomorphic to the Cartesian product on \(\widehat C\).
\end{proposition}
\begin{proof}
Co-Yoneda identifies each associativity coend with
\(C(u,a)\times C(u,b)\times C(u,c)\), and both unit coends
with \(C(u,a)\).
Use these identifications for the associator and unitors, and interchange
the two factors for the symmetry. The coherence diagrams become the
usual diagrams for Cartesian products of sets.

Since products of sets preserve colimits in each variable,
co-Yoneda also gives
\begin{align*}
 (F\starc G)(u)
 &\cong
 \left(\int^a C(u,a)\times F(a)\right)
 \times
 \left(\int^b C(u,b)\times G(b)\right)\\
 &\cong F(u)\times G(u).
\end{align*}
These isomorphisms respect the associator, unitors, and symmetry just
described. The unit \(J_C\) is the terminal presheaf.
\end{proof}

For \(C=\D\), write \(P=P_{\D}\). Thus
\(P([p],[q];-)=\D^p\times\D^q\), and \(J_{\D}=\D^0\) since
\([0]\) is terminal. In particular,
\[
 \D^p\starc\D^q\cong P([p],[q];-)=\D^p\times\D^q.
\]

The presheaf \(P([1],[1];-)\) is not representable. Indeed,
\(\D^1\times\D^1\) has four vertices, so a representing object would
have to be \([3]\). In degree one it has \(3^2=9\) simplices, whereas
\(\D^3\) has \(10\). Thus the Cartesian product on simplicial sets
need not take representables to representables.

\section{Pointwise multiplication}\label{sec:pointwise}
Fix nonnegative integers \(p,q\). For \(m\geq0\) and nondecreasing
maps \(\alpha:[m]\to[p]\) and \(\beta:[m]\to[q]\), define their
Hadamard product by
\[
 (\alpha\boxt\beta)(i)=\alpha(i)\beta(i)\qquad(0\leq i\leq m).
\]
This is a nondecreasing map to \([pq]\), because its factors are
nonnegative and nondecreasing. For every \(\theta:[\ell]\to[m]\),
\[
 (\alpha\boxt\beta)\theta=(\alpha\theta)\boxt(\beta\theta).
\]
Consequently the degreewise maps define a simplicial map
\[
 H_{p,q}:\D^p\times\D^q\longrightarrow\D^{pq}.
\]
The construction is natural in the simplicial degree.

For these fixed \(p,q\), the map \(H_{p,q}\) is a natural transformation
\(P([p],[q];-)\to\D^{pq}\).
Composing \(H_{p,q}\) with the isomorphism of
Proposition~\ref{prop:cartesian} defines the map \(\Theta_{p,q}\)
in the following commutative diagram
\[
\begin{tikzcd}[column sep=large]
 \D^p\starc\D^q \arrow[r,"\cong"] \arrow[dr,"\Theta_{p,q}"'] &
 \D^p\times\D^q \arrow[d,"H_{p,q}"]\\
 & \D^{pq}.
\end{tikzcd}
\]

For example, \(H_{2,2}\) misses the vertex \(3\) of \(\D^4\), since
\(3\) is not a product of two elements of \(\{0,1,2\}\).
Thus \(\Theta_{2,2}\) is not surjective.

\begin{remark}
There is no monoidal structure on \(\D\) whose tensor on objects is
\([p]\otimes[q]=[pq]\). Its unit would be \([1]\). Endomorphisms of a
monoidal unit commute, by the interchange law for tensor and
composition. The two constant endomorphisms \(c_0,c_1\) of \([1]\)
satisfy \(c_0c_1=c_0\ne c_1=c_1c_0\).
\end{remark}

\section{The simplicial homotopy}\label{sec:homotopy}
For simplicial maps \(f,g:X\to Y\), we use the convention that a
simplicial homotopy from \(f\) to \(g\)
is a simplicial map \(H:X\times\D^1\to Y\) with restrictions \(f\)
and \(g\) at the vertices \(0\) and \(1\), respectively;
see \cite[\S I.6]{GoerssJardine1999}.
The specialization \(q=1\) has the following endpoint behavior.

\begin{proposition}\label{prop:homotopy}
Let \(n\geq0\). For every \(m\geq0\) and nondecreasing maps
\(\alpha:[m]\to[n]\) and \(\beta:[m]\to[1]\), the formula
\[
 H^n_m(\alpha,\beta)(j)=\alpha(j)\beta(j)\qquad(0\leq j\leq m)
\]
defines a simplicial map \(H^n=H_{n,1}:\D^n\times\D^1\to\D^n\).
It is a homotopy from \(c_0\), the constant map at vertex \(0\),
to \(\id_{\D^n}\), and is constant on \(\{0\}\times\D^1\).
\end{proposition}
\begin{proof}
The pointwise product is nondecreasing, as shown in
Section~\ref{sec:pointwise}.
For every \(\theta:[\ell]\to[m]\) and \(j\in[\ell]\),
\[
 H^n_\ell(\alpha\theta,\beta\theta)(j)
 =\alpha(\theta(j))\beta(\theta(j))
 =(H^n_m(\alpha,\beta)\theta)(j).
\]
Thus the degreewise maps commute with all simplicial operators.
The vertices \(0,1\) of \(\D^1\) induce in every degree the constant
maps \([m]\to[1]\) with those values. Hence
\[
 H^n_m(\alpha,0)=0,\qquad
 H^n_m(\alpha,1)=\alpha,\qquad
 H^n_m(0,\beta)=0.
\]
These are the endpoint and fixed-vertex conditions.
\end{proof}

For \(n\geq1\) the direction matters: a simplicial homotopy from the
identity to \(c_0\) would send the edge \(\{1\}\times\D^1\) to an edge
of \(\D^n\) from vertex \(1\) to vertex \(0\), which does not exist.
After geometric realization, the interval parameter can be reversed.

\begin{remark}[Alternative construction using the nerve]
Let \(\Pos\) denote the category of posets and nondecreasing maps.
The nerve of a poset \(Q\) is the simplicial set
\(N(Q)_m=\Pos([m],Q)\), with simplicial operators given by
precomposition. In particular, \(N([n])=\D^n\).
A nondecreasing map \(f:Q\to Q'\) induces \(N(f)\)
by postcomposition; the identity \((fu)\theta=f(u\theta)\) makes
\(N(f)\) simplicial. Taking coordinate maps gives
\[
 N([n]\times[1])\cong\D^n\times\D^1.
\]
Here the product is taken in \(\Pos\), with the coordinatewise order;
for \(n\geq1\) it is not an object of \(\D\), since \((0,1)\) and
\((1,0)\) are incomparable.
Goerss and Jardine's ordered-set diagram
\cite[proof of Theorem I.7.10, p.~32]{GoerssJardine1999} specifies the
nondecreasing map
\[
 F_n:[n]\times[1]\longrightarrow[n],\qquad
 F_n(i,0)=0,\quad F_n(i,1)=i.
\]
A pair \((\alpha,\beta)\) corresponds to
\(j\mapsto(\alpha(j),\beta(j))\). Postcomposition with \(F_n\)
sends it to \(j\mapsto\alpha(j)\beta(j)\), so \(N(F_n)=H^n\)
under this identification. More generally, \(H_{p,q}\) is the nerve
of the nondecreasing map \([p]\times[q]\to[pq]\), \((i,j)\mapsto ij\).
\end{remark}

\section{Geometric realization}\label{sec:realization}
Identify the geometric realization of $\D^n$ with the standard simplex
\[
 |\D^n|=\left\{(w_0,\ldots,w_n)\in\mathbb R^{n+1}:
          w_i\geq0,\ \sum_{i=0}^n w_i=1\right\}.
\]
For $0\leq i\leq n$, let $e_i\in\mathbb R^{n+1}$ be the vector
with $1$ in coordinate $i$ and $0$ in every other coordinate.
These are the vertices of the simplex, and $w=\sum_{i=0}^n w_i e_i$
expresses a point in its barycentric coordinates.
Identify $|\D^1|$ with $[0,1]$ by $(w_0,w_1)\mapsto w_1$.
The realization maps induced by the two projections give a homeomorphism
\[
\chi_n:|\D^n\times\D^1|\longrightarrow |\D^n|\times[0,1]
\]
(see \cite[Theorem 2]{Milnor1957}). Define the realized homotopy
\[
 R_n=|H^n|\circ\chi_n^{-1}:
 |\D^n|\times[0,1]\longrightarrow|\D^n|.
\]
We describe this map on the simplices of the prism triangulation.

For $0\leq k\leq n$, let
\[
S_k=\operatorname{conv}\bigl((e_0,0),\ldots,(e_k,0),
                    (e_k,1),\ldots,(e_n,1)\bigr).
\]
Here $\operatorname{conv}$ denotes the convex hull.

\begin{proposition}\label{prop:realized_contraction}
Let $n\geq0$, $w=(w_0,\ldots,w_n)\in|\D^n|$, and $t\in[0,1]$.
Set
\[
Q_j=\sum_{i=j}^n w_i\quad(0\leq j\leq n),
\qquad Q_{n+1}=0.
\]
Thus $Q_0=1$. The barycentric coordinates of $R_n(w,t)$ are
\begin{equation}\label{eq:realized_contraction}
v_0=1-\min(t,Q_1),\qquad
v_j=\min(t,Q_j)-\min(t,Q_{j+1})\quad(1\leq j\leq n).
\end{equation}
The map $R_n$ is affine on every $S_k$ and satisfies
$R_n(w,0)=e_0$, $R_n(w,1)=w$, and $R_n(e_0,t)=e_0$.
\end{proposition}

\begin{proof}
We first locate $(w,t)$ in a prism simplex.
Since $Q_0\geq Q_1\geq\cdots\geq Q_{n+1}$, choose $k$ with
$Q_{k+1}\leq t\leq Q_k$. Then
\[
(w,t)=\sum_{i<k}w_i(e_i,0)+(Q_k-t)(e_k,0)
       +(t-Q_{k+1})(e_k,1)+\sum_{i>k}w_i(e_i,1).
\]
These nonnegative coefficients sum to one. Conversely, any convex
combination of the vertices of $S_k$ satisfies
$Q_{k+1}\leq t\leq Q_k$. Hence the $S_k$ cover the prism.

To compute $R_n$ on $S_k$, define nondecreasing maps
$\sigma_k:[n+1]\to[n]$ and $\beta_k:[n+1]\to[1]$ by
\[
\sigma_k(i)=
\begin{cases}i&i\leq k,\\ i-1&i>k,\end{cases}
\qquad
\beta_k(i)=
\begin{cases}0&i\leq k,\\ 1&i>k.\end{cases}
\]
The map $\sigma_k$ repeats $k$, and $\beta_k$ changes value between
$k$ and $k+1$. The pair $(\sigma_k,\beta_k)$ represents an
$(n+1)$-simplex of $\D^n\times\D^1$ whose geometric simplex is
mapped affinely and bijectively onto $S_k$ by $\chi_n$.
Since $|H^n|$ is affine on that simplex, $R_n$ is affine on $S_k$.
The vertex formula for $H^n$ gives
$R_n(e_i,0)=e_0$ and $R_n(e_i,1)=e_i$.
For $j\geq1$, the displayed convex combination gives $v_j=0$
when $j<k$, $v_k=t-Q_{k+1}$ when $k\geq1$, and $v_j=w_j$
when $j>k$. These values equal the differences of minima in
\eqref{eq:realized_contraction}. Summing them gives
$\sum_{j=1}^n v_j=\min(t,Q_1)$, which proves the formula for $v_0$.
The formula is independent of the choice of $k$ at a shared face.
Its endpoint values and its value at $e_0$ follow immediately.
When $n=0$, the convention $Q_1=0$ gives $v_0=1$, so the same
statement includes the one-point simplex.
\end{proof}

Consequently, $(w,t)\mapsto R_n(w,1-t)$ is a strong deformation
retraction of $|\D^n|$ onto $e_0$.

For $n=1$, the coordinate $x=w_1$ identifies the simplex with
$[0,1]$, and the realized homotopy is $R_1(x,t)=\min(x,t)$.

\section*{Acknowledgement}
The author used OpenAI's ChatGPT to assist with checking and
revising mathematical arguments, simplifying proofs, and improving
the exposition. The author takes responsibility for the mathematical
content and references.

\bibliographystyle{alpha}
\bibliography{CartesianPromonoidalKernel}
\end{document}